\documentclass[12pt]{amsart}
\usepackage{amssymb,amsmath}
\usepackage{geometry}    
\usepackage{mathrsfs}

\usepackage{vmargin}

\setmargrb{1in}{1in}{1in}{1in}

\newtheorem{theorem}{Theorem}[section]
\newtheorem{lemma}[theorem]{Lemma}

\newtheorem{corollary}[theorem]{Corollary}
\theoremstyle{definition}

\newtheorem{remark}[theorem]{Remark}

\newtheorem{conjecture}[theorem]{Conjecture}

\begin{document}

\title{On approximation constants for Liouville numbers}

\author{Johannes Schleischitz} 

\address{Institute of Mathematics, Univ. Nat. Res. Life Sci. Vienna, 1180, Austria}


\begin{abstract}
We investigate some Diophantine approximation constants 
related to the simultaneous approximation of $(\zeta,\zeta^{2},\ldots,\zeta^{k})$ for 
Liouville numbers $\zeta$. For a certain class of Liouville numbers including the
famous representative $\sum_{n\geq 1} 10^{-n!}$ and numbers in
the Cantor set, we explicitly determine all 
approximation constants simultaneously for all $k\geq 1$.
\end{abstract}

\maketitle

{\footnotesize{Supported by the Austrian Science Fund FWF grant P24828.} \\

Math Subject Classification 2010: 11J13, 11J25, 11J82}

\vspace{4mm}

\section{Definition and basic properties of $\lambda_{k,j}, \widehat{\lambda}_{k,j}$} \label{ggg}

\subsection{One dimensional approximation} \label{beginn}
We begin with the definition of the irrationality exponent.
For a real number $\zeta$ it is denoted by $\mu(\zeta)$ and defined as 
the supremum of $\eta\in{\mathbb{R}}$ for which
\begin{equation} \label{eq:1}
\left\vert \zeta-\frac{y}{x}\right\vert \leq x^{-\eta}
\end{equation}
has infinitely many solutions $(x,y)\in{\mathbb{N}\times\mathbb{Z}}$, where 
$\mathbb{N}=\{1,2,\ldots\}$. It follows from 
Dirichlet's Theorem, Corollary 2 in \cite{wald}, that
$\mu(\zeta)\geq 2$ for all $\zeta\in{\mathbb{R}}$. Capelli showed 
that actually $\mu(\zeta)=2$ for Lebesgue almost all $\zeta\in{\mathbb{R}}$, 
which was generalized first by Khintchine and later by Beresnevich, Dickinson and Velani \cite{velani}.
Roth's Theorem \cite{roth} asserts $\mu(\zeta)=2$ for all algebraic irrational real numbers $\zeta$. 
Irrational real numbers $\zeta$ with $\mu(\zeta)=\infty$ are
called Liouville numbers. Liouville's first example of a transcendental number  
\begin{equation} \label{eq:a}
L:=\sum_{n\geq 1} 10^{-n!}=10^{-1}+10^{-2}+10^{-6}+10^{-24}+\cdots
\end{equation}
is easily verified to be a Liouville number. 

\subsection{Simultaneous approximation: The $\mathbb{Q}$-linear independent case} \label{qli}
Generalizing this concept for simultaneous approximation of 
$\underline{\zeta}=(\zeta_{1},\ldots,\zeta_{k})\in{\mathbb{R}^{k}}$,
for $1\leq j\leq k+1$ define $\lambda_{k,j}(\underline{\zeta})$ resp. 
$\widehat{\lambda}_{k,j}(\underline{\zeta})$ the supremum of 
$\eta\in{\mathbb{R}}$ for which the system 
\begin{eqnarray}  
\vert x\vert &\leq& X   \label{eq:hund}  \\
\max_{1\leq j\leq k} \vert \zeta_{j}x-y_{j}\vert &\leq& X^{-\eta}  \label{eq:katz}
\end{eqnarray}
has (at least) $j$ linearly independent solutions $(x,y_{1},y_{2},\ldots, y_{k})\in{\mathbb{Z}^{k+1}}$ 
for arbitrarily large real $X$ resp. for all $X\geq X_{0}$. Note that for $k=1$, 
$\underline{\zeta}=\zeta$ the identity $\lambda_{1,1}(\zeta)+1=\mu(\zeta)$ holds. 
We should mention that the classical notation, used for example in \cite{ss}, 
for $\lambda_{k,1}(\underline{\zeta})$ resp. $\widehat{\lambda}_{k,1}(\underline{\zeta})$
is $\omega$ resp. $\widehat{\omega}$, where the vector $\underline{\zeta}$
is considered fixed and its dimension is omitted in the notation. 
There is no actual standard notation for the generalization $j\geq 2$, in~\cite{schl}
the notation $\omega_{j}$ resp. $\widehat{\omega}_{j}$ is used. 
We use the present notation because it allows convenient relation to the successive 
power case we will introduce in Section~\ref{powercase}.

Clearly, for all $\underline{\zeta}\in{\mathbb{R}^{k}}$ the inequalities
\begin{eqnarray}
\lambda_{k,1}(\underline{\zeta})&\geq&\lambda_{k,2}(\underline{\zeta})\geq 
\cdots\geq \lambda_{k,k+1}(\underline{\zeta})\geq 0,   \label{eq:dieerst} \\ 
\widehat{\lambda}_{k,1}(\underline{\zeta})&\geq& \widehat{\lambda}_{k,2}(\underline{\zeta})\geq \cdots
\geq \widehat{\lambda}_{k,k+1}(\underline{\zeta})\geq 0,   \label{eq:diezweit}
\end{eqnarray}
and 
\[
\lambda_{k,j}(\underline{\zeta})\geq \widehat{\lambda}_{k,j}(\underline{\zeta}), \qquad 1\leq j\leq k+1
\]
hold. Furthermore, a generalization of Dirichlet's~Theorem 
asserts
\begin{equation} \label{eq:tijde}
\lambda_{k,1}(\underline{\zeta})\geq \widehat{\lambda}_{k,1}(\underline{\zeta})\geq \frac{1}{k}
\end{equation}
for all $\underline{\zeta}\in{\mathbb{R}^{k}}$. 
This can be found on page~19, Chapter~2.4 in~\cite{wald}. 
Again, there is actually equality in both inequalities in \eqref{eq:tijde} 
for Lebesgue almost all $\underline{\zeta}\in{\mathbb{R}^{k}}$, 
such as for all algebraic $\underline{\zeta}\in{\overline{\mathbb{Q}}^{k}}$ 
for which $\{1,\zeta_{1},\ldots,\zeta_{k}\}$ is $\mathbb{Q}$-linearly independent. 
The first assertion was established by Khintchine,
the latter is a consequence of Schmidt's Subspace~Theorem~\cite{schm}.
Define
\begin{eqnarray}
\chi_{k,1}=\chi_{k,2}&=&\frac{1}{k}, \qquad \qquad \chi_{k,3}=\chi_{k,4}=\cdots=\chi_{k,k+1}=0,  \label{eq:chi}  \\
\phi_{k,1}&=&\frac{1}{k}, \qquad\qquad \phi_{k,2}=\phi_{k,3}=\cdots=\phi_{k,k+1}=0.  \nonumber
\end{eqnarray}
For $\underline{\zeta}\in{\mathbb{R}^{k}}$ that satisfies the $\mathbb{Q}$-linear independence property above,
(14)-(19) in~\cite{schl} (where $\omega_{j}$ corresponds to the present $\lambda_{k,j}(\underline{\zeta})$)
translate into
\begin{eqnarray}
\chi_{k,j}\leq \lambda_{k,j}(\underline{\zeta})&\leq& \frac{1}{j-1}, \qquad 1\leq j\leq k+1,  \label{eq:2} \\
\phi_{k,j}\leq \widehat{\lambda}_{k,j}(\underline{\zeta})&\leq& \frac{1}{j}, \qquad\qquad 1\leq j\leq k,  \label{eq:3}\\
\phi_{k,k+1}\leq\widehat{\lambda}_{k,k+1}(\underline{\zeta})&\leq& \frac{1}{k}.  \label{eq:4}
\end{eqnarray}
Here we have put $1/0=\infty$, and we agree on this such as $1/\infty=0$ anywhere it appears. 
Individually, these bounds are best possible, as deduced in Corollary~9 in~\cite{schl}.   
In fact, the linear independence property is required only in \eqref{eq:3}.
Several other restrictions concerning the {\em joint} spectrum
of the quantities $\lambda_{k,j}$ resp. $\widehat{\lambda}_{k,j}$
are known, mostly inferred via Minkowski's second lattice point Theorem~\cite{minkowski}, 
see~\cite{ss}, \cite{ssch}. 

\subsection{Simultaneous approximation: The successive power case}  \label{powercase}
An interesting heavily studied special case is that $\underline{\zeta}$ as in Section~\ref{qli}
consists of the first successive powers of a real number $\zeta$, 
i.e. $\underline{\zeta}=(\zeta,\zeta^{2},\ldots,\zeta^{k})$.
We will mostly deal with this case in the sequel, so for simplicity for any pair of positive 
integers $k,j$ with $1\leq j\leq k+1$ we will write 
\[
\lambda_{k,j}(\zeta):=\lambda_{k,j}(\zeta,\zeta^{2},\ldots,\zeta^{k}), \qquad  
\widehat{\lambda}_{k,j}(\zeta):=\widehat{\lambda}_{k,j}(\zeta,\zeta^{2},\ldots,\zeta^{k}).
\]
Again we also assume the $\mathbb{Q}$-linear independence condition from the previous 
Section~\ref{qli} to hold, which has the natural interpretation 
that $\zeta$ is not algebraic of degree $\leq k$.
  
Lebesgue almost all $\zeta$ satisfy $\lambda_{k,j}(\zeta)=\widehat{\lambda}_{k,j}(\zeta)=1/k$ 
for all pairs $j,k$ with $1\leq j\leq k+1$. This
was established by Sprind\^zuk~\cite{sprindzuk} for $j=1$, 
combination with Minkowski's second lattice point Theorem readily yields the general assertion.
Generalizations of this metric result for $\underline{\zeta}=(\zeta_{1},\zeta_{2},\ldots,\zeta_{k})$
belonging to a wider class of algebraic curves in $\mathbb{R}^{k}$ can be found in~\cite{kleinbock}. 

Clearly all inequalities and bounds established in Section~\ref{qli} hold
for the special case of successive powers as well. 
However, the bounds \eqref{eq:2}-\eqref{eq:4}
are in general not optimal, particularly those related to the {\em uniform} approximation 
constants (with the ''hat'') differ significantly. 
This is a consequence of the upper bound
\begin{equation} \label{eq:piscis}
\widehat{\lambda}_{k,1}(\zeta)\leq \left\lceil \frac{k}{2}\right\rceil^{-1}
\end{equation}
valid for all $\zeta$ not algebraic of degree~$\leq \lceil k/2\rceil$
due to Davenport, Schmidt and Laurent~\cite{laurent}. 
Here as usual $\lceil \alpha\rceil$ denotes the smallest integer greater or equal to $\alpha\in{\mathbb{R}}$.
Indeed $\lceil k/2\rceil^{-1}$ sharpens the bound $1$ for $\widehat{\lambda}_{k,1}(\zeta)$ arising 
from \eqref{eq:3} vastly, and \eqref{eq:diezweit} further implies sharper restrictions on
the spectrum of $\widehat{\lambda}_{k,j}(\zeta)$ at least for some indices $2\leq j\leq k+1$ too. 

The spectrum of $\lambda_{k,j}(\zeta)$ among all real numbers $\zeta$ in dependence 
of $k$ has been studied as well, mostly for $j=1$. 
Bugeaud proved in Theorem~2 in \cite{bug} that the spectrum of $\lambda_{k,1}$ contains $[1,\infty]$,
a generalization due to the author is given in~\cite{joxy}.   
The natural conjecture already quoted in~\cite{bug} is the following.
\begin{conjecture}  \label{conj}
Let $k$ be a positive integer. The spectrum of $\lambda_{k,1}(\zeta)$ among 
all real numbers $\zeta$ not algebraic of degree $\leq k$ is $[1/k,\infty]$.
\end{conjecture}
Apart from $k=1$ this is only known for $k=2$
thanks to Beresnevich, Dickinson, Vaughan and Velani~\cite{bere},~\cite{vel}. 
For further problems and results in this manner see also~\cite{bug}.
We will return to the spectrum of $\lambda_{k,j}(\zeta), \widehat{\lambda}_{k,j}(\zeta)$ 
for fixed $k$ in Section~\ref{sek3}.

It follows from their definition that for fixed $\zeta$ and $j\geq 1$, 
the quantities $\lambda_{k,j}(\zeta), \widehat{\lambda}_{k,j}(\zeta)$ are 
non-increasing as $k$ increases, i.e. (with $\lambda_{0,1}(\zeta):=\infty, \widehat{\lambda}_{0,1}(\zeta):=\infty$)
\begin{eqnarray} 
\lambda_{j-1,j}(\zeta)&\geq& 
\lambda_{j,j}(\zeta)\geq \lambda_{j+1,j}(\zeta)\geq \cdots, \qquad\quad j\geq 1, \label{eq:trug}  \\
\widehat{\lambda}_{j-1,j}(\zeta)&\geq& \widehat{\lambda}_{j,j}(\zeta)
\geq \widehat{\lambda}_{j+1,j}(\zeta)\geq\cdots, \qquad\quad j\geq 1.   \label{eq:truge}
\end{eqnarray}  
For $j=1$, Lemma~1 in~\cite{bug} by Bugeaud states a result
somehow reverse to \eqref{eq:trug}, which in our notation says the following. 

\begin{theorem}[Bugeaud] \label{th}
For any positive integers $k$ and $n$, and any transcendental real number $\zeta$ we have
\[
\lambda_{kn,1}(\zeta)\geq \frac{\lambda_{n,1}(\zeta)-k+1}{k}.
\]
\end{theorem}
There is actually equality in case of $\lambda_{kn,1}(\zeta)>1$, as established 
by the author~\cite{schleischitz}.
If we put $n=1$ in Theorem~\ref{th} and let $\zeta$ be a Liouville number, we obtain
Corollary~2 in~\cite{bug}.
\begin{corollary}[Bugeaud]  \label{koro}
Let $\zeta$ be an irrational real number. We have
$\lambda_{k,1}(\zeta)=\infty$ for all positive integers $k$ if and only if $\lambda_{1,1}(\zeta)=\infty$,
i.e. $\zeta$ is a Liouville number.
\end{corollary}
We shall utilize Corollary~\ref{koro} in the proof of Theorem~\ref{grundlage}.

\section{The spectrum of $\lambda_{k,j}(\zeta), \widehat{\lambda}_{k,j}(\zeta)$: known facts}  \label{sek3}

In this section we quote facts on the individual such as the joint spectrum of
the quantities $\lambda_{k,j}(\zeta), \widehat{\lambda}_{k,j}(\zeta)$.
Let us restrict to $j=1$ first.
The restriction \eqref{eq:piscis} shows that for not too small values of $\lambda_{k,1}(\zeta)$,
there are far more stringent restrictions for the joint spectrum of
($\lambda_{k,1}(\zeta),\widehat{\lambda}_{k,1}(\zeta)$) than \eqref{eq:tijde}.
Another result that affirms this arises from the following Theorem, which we deduce at once 
from the Theorems~1.6,~1.12 in~\cite{schleischitz} due to the author.

\begin{theorem}[S, 2014]   \label{lemma6}
Let $\zeta\in{\mathbb{R}\setminus{\mathbb{Q}}}$. For any positive integer $k$, we have
\[
 \widehat{\lambda}_{k,1}(\zeta)\leq \max\left\{\frac{1}{\lambda_{1,1}(\zeta)},\frac{1}{k}\right\}.
\]
Moreover, if $\lambda_{k,1}(\zeta)>1$, then $\lambda_{1,1}(\zeta)=k\lambda_{k,1}(\zeta)+k-1>2k-1\geq k$.
\end{theorem}

We are again particularly interested in the consequences for Liouville numbers.
The following corollary is basically contained in Corollary~5.2 in~\cite{schleischitz}.

\begin{corollary} \label{wiedenn}
Let $k$ be a positive integer.
Let $\zeta\in{\mathbb{R}\setminus{\mathbb{Q}}}$ such that $\lambda_{1,1}(\zeta)\geq k$, 
which is in particular true if $\lambda_{k,1}(\zeta)>1$. Then
\[
 \widehat{\lambda}_{k,1}(\zeta)= \frac{1}{k}.
\]
In particular, if $\zeta$ is a Liouville number, then $ \widehat{\lambda}_{k,1}(\zeta)= 1/k$
simultaneously for all $k$. 
\end{corollary}

\begin{proof}
Combination of \eqref{eq:tijde} and Theorem~\ref{lemma6}.
\end{proof}

We should mention that for $k=2$, 
the maximum value of $\widehat{\lambda}_{2,1}(\zeta)$ among 
all $\zeta$ not algebraic of degree $\leq 2$ is known. It is attained for
so called extremal numbers $\zeta$, which satisfy
$\lambda_{2,1}(\zeta)=1, \widehat{\lambda}_{2,1}(\zeta)=\gamma:=(\sqrt{5}-1)/2$, see~\cite{roy}.
Besides, extremal numbers provide explicit vectors $\underline{\zeta}=(\zeta,\zeta^{2})$ where equality holds
in inequality (1.21) in~\cite{ssch} for $k=2$, which is in our notation
\begin{equation} \label{eq:sumsch}
\lambda_{k,1}(\underline{\zeta})\geq 
\frac{\widehat{\lambda}_{k,1}(\underline{\zeta})^{2}+(k-2)\widehat{\lambda}_{k,1}(\underline{\zeta})}
{(k-1)(1-\widehat{\lambda}_{k,1}(\underline{\zeta}))}.
\end{equation}
Note that \eqref{eq:sumsch} is valid for any $k\geq 2$ and $\underline{\zeta}\in{\mathbb{R}^{k}}$ 
which is $\mathbb{Q}$-linearly independent together with $1$. More generally,
an extremal number induces a certain geometric pattern which Summerer and Schmidt call a regular graph, 
see~\cite{sums}. See also their papers \cite{ss}, \cite{ssch}
such as Roy's recent papers \cite{royyyy}, \cite{roy5}, \cite{droy} for the (joint) spectrum of
$\lambda_{k,j}(\underline{\zeta}), \widehat{\lambda}_{k,j}(\underline{\zeta})$
in the $\mathbb{Q}$-linearly independent case of Section~\ref{qli}.

Observe that for $k=2$ the estimate \eqref{eq:piscis} would only yield the trivial 
upper bound $1$ from \eqref{eq:3} that holds for any $\underline{\zeta}$ under the linear 
independence assumption, which is significantly larger than $\gamma\approx 0.6180$. 
For $k=3$, the bound $1/2$ from \eqref{eq:piscis} was improved as well by Roy~\cite{roy2}, however
the best value remains unknown. For $k\geq 4$, no improvements concerning
the spectrum of the constants $\lambda_{k,1}(\zeta), \widehat{\lambda}_{k,1}(\zeta)$
in this spirit have been made.

The quantities $\lambda_{k,j}(\zeta), \widehat{\lambda}_{k,j}(\zeta)$ for
$2\leq j\leq k+1$ remain very poorly studied.
The remainder of this paper is devoted
to this question in the special case of Liouville numbers~$\zeta$.
It will turn out in Theorem~\ref{grundlage}
that for Liouville numbers all uniform constants $\widehat{\lambda}_{k,j}(\zeta)$
can be explicitly determined. Moreover, Theorem~\ref{mainth} will suggest Conjecture~\ref{konjw} 
on the (individual) spectrum of each $\lambda_{k,j}(\zeta)$ for $\zeta$ not algebraic of
degree $\leq k$.

\section{Formulation of the main results} \label{sek4}

\subsection{The general case}

The concern of the first theorem is to determine/bound 
the classic approximation constants for arbitrary Liouville numbers.

\begin{theorem}  \label{grundlage}
Let $\zeta$ be a Liouville number. For any positive integer $k$, we have
\begin{eqnarray}
\lambda_{k,1}(\zeta)&=&\infty,     \label{eq:w} \\
\frac{1}{k}\leq \lambda_{k,j}(\zeta)&\leq& \frac{1}{j-1}, \qquad 2\leq j\leq k+1, \label{eq:x}  \\
\widehat{\lambda}_{k,1}(\zeta)&=&\frac{1}{k},    \label{eq:y}  \\
\widehat{\lambda}_{k,j}(\zeta)&=&0, \qquad\qquad 2\leq j\leq k+1.  \label{eq:z}
\end{eqnarray}
\end{theorem}  

\subsection{A special class of Liouville numbers} 

We provide a class of Liouville numbers $\zeta$, for which all approximation constants 
$\lambda_{k,j}(\zeta), \widehat{\lambda}_{k,j}(\zeta)$ can be determined explicitly
simultaneously for all $k$. The weakest assumptions on $\zeta$ for our methods to work
will be those in Theorem~\ref{mainth}. Various specializations will be given in Section~\ref{specgen}. 

\begin{theorem}\label{mainth}
Let $k$ be a positive integer. Further let $q_{1},q_{2},\ldots$ be positive integers
such that $q_{l}\vert q_{l+1}$ and $q_{l}<q_{l+1}$ for all $l\geq 1$. 
Suppose that for a positive integer $k$
and all $C>0$ there exists $n=n(C)$ such that
\begin{equation} \label{eq:schwaechst}
 \frac{\log q_{n+1}}{\log q_{n}} > C, \qquad \frac{\log q_{n+2}}{\log q_{n+1}} > k+1.
\end{equation}
Further set
\[
\zeta=\sum_{l\geq 1} \frac{1}{q_{l}}.
\]
Then we have
\begin{eqnarray*}
 \lambda_{k,j}(\zeta)&=& \frac{1}{j-1}, \qquad 1\leq j\leq k+1,  \\
 \widehat{\lambda}_{k,1}(\zeta)&=& \frac{1}{k},                                \\
 \widehat{\lambda}_{k,j}(\zeta)&=& 0, \qquad 2\leq j\leq k+1.
\end{eqnarray*}
\end{theorem}

We will proof both Theorem~\ref{grundlage},~\ref{mainth} and give some remarks in 
Section~\ref{theproofs}, and will discuss some consequences in Section~\ref{specgen}.

\section{Preparatory results for the proofs}\label{sek2}

Parts of the proofs of the main new results in Section~\ref{sek4} can be generalized, so we will
state the more general versions in form of two Lemmas~\ref{helplemma},~\ref{helplemma2}. 
The proofs of those lemmas
are basically a consequence of Minkowski's second lattice point Theorem, hidden in the results
from~\cite{ss},~\cite{ssch} we utilize.

We introduce the functions $\psi_{k,j}$ and the derived quantities 
$\underline{\psi}_{k,j}, \overline{\psi}_{k,j}$ defined in~\cite{ss} with 
a slightly different notation (subindex $k$ omitted).
For fixed $\underline{\zeta}=(\zeta_{1},\ldots,\zeta_{k})\in{\mathbb{R}^{k}}$ 
and a real parameter $Q>1$, for $1\leq j\leq k+1$ define $\psi_{k,j}(Q)$ as the 
supremum of all values of $\nu\in{\mathbb{R}}$ such that the system 
\begin{eqnarray*}
\vert x\vert &\leq& Q^{1+\nu}   \\
\max_{1\leq j\leq k}\vert \zeta_{j} x-y_{j}\vert &\leq& Q^{-\frac{1}{k}+\nu} 
\end{eqnarray*}
has (at least) $j$ linearly independent solutions $(x,y_{1},\ldots,y_{k})\in{\mathbb{Z}^{k+1}}$. Further let 
\[
\underline{\psi}_{k,j}=\liminf_{Q\to\infty} \psi_{k,j}(Q), 
\quad \overline{\psi}_{k,j}=\limsup_{Q\to\infty} \psi_{k,j}(Q).
\]
The quantities obviously have the properties $\underline{\psi}_{k,j}\leq \overline{\psi}_{k,j}$ and
\begin{eqnarray}
-1\leq \underline{\psi}_{k,1}&\leq& \underline{\psi}_{k,2}\leq \cdots\leq \underline{\psi}_{k,k+1}\leq 1/k,
 \label{eq:oben}  \\
-1\leq \overline{\psi}_{k,1}&\leq& \overline{\psi}_{k,2}\leq \cdots\leq \overline{\psi}_{k,k+1}\leq 1/k.
\label{eq:unten}
\end{eqnarray}
Minkowski's first lattice point Theorem (or Dirichlet's Theorem) further implies
\begin{equation} \label{eq:mink}
\overline{\psi}_{k,1}\leq 0.
\end{equation}
Moreover, the generalization (13) in~\cite{schl} of (1.9) in~\cite{ss} for arbitrary $j$, 
shows that $\underline{\psi}_{k,j}$ and $\lambda_{k,j}=\lambda_{k,j}(\underline{\zeta})$ such
as $\overline{\psi}_{k,j}$ and $\widehat{\lambda}_{k,j}=\widehat{\lambda}_{k,j}(\underline{\zeta})$ 
are closely connected. In our notation the correspondence is given by
\begin{equation} \label{eq:alt}
(1+\lambda_{k,j})(1+\underline{\psi}_{k,j})=(1+\widehat{\lambda}_{k,j})(1+\overline{\psi}_{k,j})=\frac{k+1}{k}, 
\qquad 1\leq j\leq k+1.
\end{equation}
We will use \eqref{eq:alt} frequently in the proof of the following 
Lemmas~\ref{helplemma},~\ref{helplemma2}. Moreover, we will utilize
(1.10) in~\cite{ssch}, which in the notation of the present Section~\ref{sek2} states
\begin{align} 
j\underline{\psi}_{k,1}+(k+1-j)\overline{\psi}_{k,j+1}&\leq 0, \qquad 0\leq j\leq k,  \label{eq:endli} \\
j\overline{\psi}_{k,1}+(k+1-j)\underline{\psi}_{k,j+1}&\leq 0, \qquad 0\leq j\leq k.  \label{eq:endlix}
\end{align}
Further we need
\begin{equation} \label{eq:freckle}
\overline{\psi}_{k,1}+k\underline{\psi}_{k,k+1}\geq 0,
\end{equation}
which is the right estimate in (1.13) in~\cite{ssch}. Finally we need
inequality (1.6a) in~\cite{ss} for $i=1$, which is in the present notation
\begin{equation} \label{eq:altespaper}
\underline{\psi}_{k,1}+\overline{\psi}_{k,2}+\overline{\psi}_{k,3}+\cdots+\overline{\psi}_{k,k+1}\geq 0.
\end{equation}
Before we state and prove the lemmas,
we should add that the functions $\psi_{k,j}(Q)$ and thus the derived values 
$\underline{\psi}_{k,j}, \overline{\psi}_{k,j}$ have a natural geometric interpretation
as successive minima of special convex bodies with respect to a lattice.
However, we will not explicitly use the geometric view for our approaches and refer 
to~\cite{ss} for details.

\begin{lemma}\label{helplemma}
For any positive integer $k$ and any $\underline{\zeta}\in{\mathbb{R}^{k}}$ we have the equivalence
\begin{equation} \label{eq:spaeter}
\lambda_{k,1}(\underline{\zeta})=\infty \quad \Longleftrightarrow \quad \widehat{\lambda}_{k,2}(\underline{\zeta})=
\widehat{\lambda}_{k,3}(\underline{\zeta})=\cdots=\widehat{\lambda}_{k,k+1}(\underline{\zeta})=0.
\end{equation}
\end{lemma}

\begin{proof}
Actually, the author basically
deduces the left to right implication of \eqref{eq:spaeter} in the last part of the 
proof of Theorem~1 in~\cite{schl} with the aid of the quantities $\underline{\psi}_{k,j}, \overline{\psi}_{k,j}$. 
We want to prove it again rigorously, however. 
With \eqref{eq:alt} one can reinterpret $\lambda_{k,1}(\underline{\zeta})=\infty$ as $\underline{\psi}_{k,1}=-1$.
Together with inequality \eqref{eq:altespaper} and noting $\overline{\psi}_{j}\leq 1/k$ 
for all $1\leq j\leq k+1$ by \eqref{eq:unten}, we conclude 
$\overline{\psi}_{k,2}=\cdots=\overline{\psi}_{k,k+1}=1/k$. Again reinterpreting this
using \eqref{eq:alt}, the right hand side of \eqref{eq:spaeter} follows.
We prove the other direction in \eqref{eq:spaeter}. 
First note that \eqref{eq:alt} with $j=2$ and the right hand side of \eqref{eq:spaeter}
yields $\overline{\psi}_{k,2}=1/k$. Applying \eqref{eq:endli} for $j=1$ 
further gives $\underline{\psi}_{k,1}+1=\underline{\psi}_{k,1}+k\overline{\psi}_{k,2}\leq 0$.
Hence \eqref{eq:oben} implies $\underline{\psi}_{k,1}=-1$, and again with \eqref{eq:alt}
for $j=1$ finally $\lambda_{k,1}(\underline{\zeta})=\infty$. 
\end{proof}

\begin{lemma} \label{helplemma2}
Let $k$ be a positive integer and $\underline{\zeta}\in{\mathbb{R}^{k}}$ arbitrary. Then 
\[
\widehat{\lambda}_{k,1}(\underline{\zeta})=\frac{1}{k} \quad \Longleftrightarrow \quad 
\lambda_{k,2}(\underline{\zeta})\geq \lambda_{k,3}(\underline{\zeta})\geq \cdots\geq
\lambda_{k,k+1}(\underline{\zeta})= \frac{1}{k}.
\]
\end{lemma}

\begin{proof}
First note that by \eqref{eq:dieerst}, it remains to show the
equivalence of the left hand side to the last equality on the right hand side. 
The implication left to right of the lemma is a consequence of \eqref{eq:endlix}.
By \eqref{eq:alt} we know $\widehat{\lambda}_{k,1}(\underline{\zeta})=1/k$ is equivalent 
to $\overline{\psi}_{k,1}=0$, hence \eqref{eq:endlix} implies
$\underline{\psi}_{k,j}\leq 0$ for all $1\leq j\leq k+1$, and again with \eqref{eq:alt}
we conclude $\lambda_{k,k+1}(\underline{\zeta})\geq 1/k$. The reverse inequality 
$\lambda_{k,k+1}(\underline{\zeta})\leq 1/k$ is by \eqref{eq:alt} equivalent to $\underline{\psi}_{k,k+1}\geq 0$,
but \eqref{eq:freckle} and \eqref{eq:mink} indeed imply
\begin{equation} \label{eq:zwamoi}
\underline{\psi}_{k,k+1}\geq -\frac{1}{k}\overline{\psi}_{k,1}\geq 0.
\end{equation}
We turn to the implication right to left. Suppose $\lambda_{k,k+1}(\underline{\zeta})= 1/k$, which
by \eqref{eq:alt} is equivalent to $\underline{\psi}_{k,k+1}=0$. On the other hand,
by virtue of \eqref{eq:alt} the assertion $\widehat{\lambda}_{k,1}(\underline{\zeta})=1/k$ 
is equivalent to $\overline{\psi}_{k,1}=0$. However, this is a consequence of 
$\underline{\psi}_{k,k+1}=0$, \eqref{eq:mink} and \eqref{eq:zwamoi}.
\end{proof}

\section{Proofs of Theorems~\ref{grundlage}, \ref{mainth} and consequences} \label{consequences}
  
\subsection{Proofs} \label{theproofs}
Theorem~\ref{grundlage} is readily inferred by combining the results of
Section~\ref{sek3} with the results we have established
in Section~\ref{sek2}. 
	
\begin{proof}[Proof of Theorem~\ref{grundlage}]
Corollary~\ref{koro} yields \eqref{eq:w} and Lemma~\ref{helplemma} implies \eqref{eq:z}. 
The upper bounds in \eqref{eq:x} are
due to \eqref{eq:2}, since the $\mathbb{Q}$-linear independence condition
is certainly satisfied as Liouville numbers are transcendental
by Roth's Theorem, see Section~\ref{beginn}.
The lower bounds in \eqref{eq:x} are due to Lemma~\ref{helplemma2}.
Finally Corollary~\ref{wiedenn} gives \eqref{eq:y}. 
\end{proof}

We turn to the more technical proof of Theorem~\ref{mainth}.
	
\begin{proof}[Proof of Theorem~\ref{mainth}]
The conditions on the sequence $(q_{l})_{l\geq 1}$ imply $q_{l+1}\geq 2q_{l}$ and thus
the sum converges and $\zeta$ as in the theorem is well-defined.
It is further easily checked by means of \eqref{eq:schwaechst} that $\zeta$ is a Liouville number. 
In view of Theorem~\ref{grundlage}, it suffices to prove 
\begin{equation} \label{eq:reicht}
\lambda_{k,j}(\zeta)\geq \frac{1}{j-1}, \qquad 2\leq j\leq k+1.
\end{equation}
Suppose that for each $C>0$ there exists a large index
$n=n(C)$ for which \eqref{eq:schwaechst} is satisfied.
We consider the pair $C,n$ fixed in the sequel. 
Define the integers
\[
U=q_{n}, \qquad V=q_{n+1},\qquad A=q_{n}\sum_{l=1}^{n}\frac{1}{q_{l}}.
\]
For each pair $(j,m)\in{\{1,2,\ldots,k+1\}^{2}}$ denote $E_{j,m}$ the closest 
integer to $U^{k}V^{j-1}\zeta^{m-1}$ and let
\[
E_{j}=(E_{j,1},\ldots,E_{j,k+1}), \qquad 1\leq j\leq k+1.
\]
In view of the condition $\log q_{n+2}/\log q_{n+1}>k+1$,
we see that $\zeta$ is very close to $U^{-1}A+V^{-1}$, more
precisely
\begin{equation} \label{eq:nonedda}
\vert \zeta-(U^{-1}A+V^{-1})\vert \ll V^{-k-1}. 
\end{equation}
From \eqref{eq:nonedda} it follows that actually
\begin{equation} \label{eq:baldur}
\vert \zeta^{m}-(U^{-1}A+V^{-1})^{m}\vert \ll V^{-k-1}, \qquad 1\leq m\leq k.
\end{equation}
Hence, for any pair $(j,m)\in{\{1,2,\ldots,k+1\}^{2}}$, one checks that $E_{j,m}$ 
is also the closest integer to $U^{k}V^{j-1}(U^{-1}A+V^{-1})^{m-1}$. 
This further implies
\begin{equation} \label{eq:krass}
E_{j,m}=
U^{k}V^{j-1}\sum_{i=0}^{\min\{m-1,j-1\}} \binom{m-1}{i} \left(\frac{A}{U}\right)^{m-1-i}\left(\frac{1}{V}\right)^{i},
\end{equation}
as the right hand side is an integer and the omitted terms in the expansion 
of $U^{k}V^{j-1}(U^{-1}A+V^{-1})^{m-1}$ sum up to a very small number, clearly smaller than $1/2$. 
More precisely, from the definition of $E_{j,m}$ and \eqref{eq:krass}, we deduce 
\begin{equation} \label{eq:angel}
\vert E_{j,m}-U^{k}V^{j-1}\zeta^{m-1}\vert \!\ll\! U^{k}V^{j-1}V^{-j}
= U^{k}V^{-1}, \qquad 1\leq j,m\leq k+1.
\end{equation}
Observe that $E_{j,1}=U^{k}V^{j-1}\geq V^{j-1}$ and moreover
\eqref{eq:schwaechst} implies $U^{k}V^{-1}\ll V^{k/C-1}$. 
Hence \eqref{eq:angel} further yields
\begin{equation} \label{eq:uvbed}
\max_{1\leq i\leq k}-\frac{\log \Vert E_{j,1}\zeta^{i}\Vert}{\log E_{j,1}}\geq 
\frac{1-\frac{k}{C}+\frac{D}{\log V}}{j-1}=:H_{j,C},
\end{equation}
with a constant $D=D(k,\zeta)$ independent from $C,n$.
Moreover, since
\[
E_{1,1}<E_{2,1}<\cdots<E_{k+1,1},
\]
for any $1\leq j\leq k+1$ the first coordinate of any of the first $j$ integer 
vectors $\{E_{1},E_{2},\ldots,E_{j}\}$ is bounded above by $E_{j,1}$.
We conclude from \eqref{eq:uvbed} that the system \eqref{eq:hund},\eqref{eq:katz}
with respect to this value $X:=E_{j,1}$ and the $j$ vectors $(x,y_{1},\ldots,y_{k}):=E_{i}$ 
for $1\leq i\leq j$, will be satisfied for a value $\eta\geq H_{j,C}$.
Now observe that $H_{j,C}$ will be arbitrarily close to $1/(j-1)$ if we let $C\to\infty$,
since then both $k/C$ and $D/\log V$ tend to $0$. 

It remains to prove the linear independence of $E_{1},E_{2},\ldots,E_{k+1}$
to deduce \eqref{eq:reicht}. 
This is equivalent to the regularity of the matrix $(E_{j,m})_{1\leq j,m\leq k+1}$.
However, by virtue of \eqref{eq:krass} the reduction of $(E_{j,m})_{1\leq j,m\leq k+1}$ modulo $V$ 
leads to an upper triangular matrix with $U^{k}$ everywhere on the diagonal, consequently its determinant
in $\mathbb{Z}/V\mathbb{Z}$ equals $U^{k(k+1)}$. Provided that $C>k(k+1)$ we have $0<U^{k(k+1)}<V$
such that this determinant does not vanish, so clearly the determinant of $(E_{j,m})_{1\leq j,m\leq k+1}$
is non-zero as well. 
\end{proof}

\begin{remark}
For $j>1$ the quantities $\lambda_{k,j}(\zeta)$ need not equal the upper bound
$1/(j-1)$ for {\em any} Liouville number $\zeta$ (not of the
form as in Theorem~\ref{mainth}). This will most likely be wrong for many Liouville numbers 
that are normal to any base $b\geq 2$, whose existence was established in~\cite{bugeaud}.
Theorem~\ref{mainth} suggests that there should be equality
for some Liouville numbers normal to any base, though, because for generic choices of the $q_{l}$
there is no reason for the resulting Liouville number not to be normal in any base.
\end{remark}

\begin{remark} \label{rem4}
Roughly speaking, the approximation functions $\psi_{k,j}(Q)$ introduced in Section~\ref{sek2},
related to $(\zeta,\zeta^{2},\ldots,\zeta^{k})$ 
with $\zeta$ as in Theorem~\ref{mainth}, all have the maximum possible range,
that is $(-1,0)$ for $j=1$ and $((j-k-1)/kj,1/k)$ for $2\leq j\leq k+1$, in
any interval $(Q_{0},\infty)$.
This case was already constructed in Corollary~3 in~\cite{schl} for $\underline{\zeta}\in{\mathbb{R}^{k}}$
that satisfy the $\mathbb{Q}$-linear independence property, however Theorem~\ref{mainth} 
provides examples for the special case $\underline{\zeta}=(\zeta,\zeta^{2},\ldots,\zeta^{k})$.
\end{remark}

\subsection{Specifications} \label{specgen}
The most obvious way to specify Theorem~\ref{mainth} is
to put $q_{l}=b^{a_{l}}$ for some fixed integer $b\geq 2$.

\begin{corollary} \label{vereinfachen}
Let $k$ be a positive integer and $(a_{l})_{l\geq 1}$ be a strictly increasing
sequence of positive integers with the property 
\begin{equation} \label{eq:bedingung}
\limsup_{n\to\infty} \min\left\{\frac{a_{n+1}}{a_{n}},\frac{a_{n+2}}{a_{n+1}}\right\}=\infty.
\end{equation}
Let $b\geq 2$ be an integer and let $\zeta=\sum_{l\geq 1} b^{-a_{l}}$. 
Then the quantities $\lambda_{k,j}(\zeta), \widehat{\lambda}_{k,j}(\zeta)$ are
given as in Theorem~{\upshape\ref{mainth}}. 
\end{corollary}
\begin{proof}
Apply Theorem~\ref{mainth} with $q_{l}:=b^{a_{l}}$ and note that \eqref{eq:schwaechst}
follows from \eqref{eq:bedingung}.
\end{proof}

\begin{remark}
The condition \eqref{eq:bedingung} of Corollary~\ref{vereinfachen} is in particular satisfied if 
\begin{equation} \label{eq:2mal}
\lim_{n\to\infty}\frac{a_{n+1}}{a_{n}}=\infty.
\end{equation}
\end{remark}

We want to specify Corollary~\ref{vereinfachen} further in two ways. 
\begin{corollary}    \label{kor8}
Let $L$ be as in {\upshape(\ref{eq:a})}. Then for any positive integer $k$ and $1\leq j\leq k+1$
the quantities $\lambda_{k,j}(L),\widehat{\lambda}_{k,j}(L)$
are given as in Theorem~{\upshape\ref{mainth}}.
\end{corollary}

\begin{proof}
Let $b=10$ and $a_{l}=l!$ in Corollary~\ref{vereinfachen} and observe that \eqref{eq:2mal} holds.
\end{proof}

Recall the Cantor middle third set is the set of numbers in $[0,1]$ which can
be written without the digit $1$ in the $3$-adic expansion. 

\begin{corollary} \label{remarkerl}
The number $\zeta=2\cdot \sum_{l\geq 1} 3^{-l!}$ is an element of
the Cantor middle third set, whose approximation constants
$\lambda_{k,j}(\zeta),\widehat{\lambda}_{k,j}(\zeta)$ are given as in Theorem~{\upshape\ref{mainth}}.
\end{corollary}

\begin{proof}
Let $b=3$ and $a_{l}=l!$, then \eqref{eq:2mal} holds and we may apply
Corollary~\ref{vereinfachen} to $\zeta/2$.
Finally rational transformations do not affect the quantities 
$\lambda_{k,j}(\zeta),\widehat{\lambda}_{k,j}(\zeta)$.
\end{proof}
Recall $\chi_{k,j}$ from \eqref{eq:chi}. Theorem~\ref{mainth}, Corollary~\ref{remarkerl} and 
the bounds of $\lambda_{k,j}(\zeta)$ from \eqref{eq:2} 
suggest the following generalizations of Conjecture~\ref{conj}. 
\begin{conjecture}[Weak]   \label{konjw}
Let $k$ be a positive integer and $1\leq j\leq k+1$. The (individual) spectrum of $\lambda_{k,j}(\zeta)$ among 
all real $\zeta$ not algebraic of degree $\leq k$ is $[\chi_{k,j},1/(j-1)]$.

\end{conjecture}
\begin{conjecture}[Strong] 
Let $k$ be a positive integer and $1\leq j\leq k+1$. The (individual) spectrum of $\lambda_{k,j}(\zeta)$ among 
irrational $\zeta$ in the Cantor set is $[\chi_{k,j},1/(j-1)]$.
\end{conjecture}
Recall that the analogue of the conjectures for the uniform constants $\widehat{\lambda}_{k,j}(\zeta)$
with the bounds from \eqref{eq:3},\eqref{eq:4} fails heavily due to \eqref{eq:piscis}.

At this point we want to add that unmentioned there,
similar to Corollary~\ref{remarkerl} one can choose numbers in the Cantor middle third set
satisfying the results in Section~2 in~\cite{schl} by small rearrangements.
A concrete example is changing the base from $2$ to $3$ and 
replacing each $\zeta_{j}$ by $2\zeta_{j}$ in Corollary~7 in~\cite{schl}.

Eventually, we want to point out that the results in Section~\ref{sek4}
have an equivalent interpretation in a well-studied linear form problem.
Define $w_{k,j}(\zeta)$ resp. $\widehat{w}_{k,j}(\zeta)$ as the 
supremum of real $\nu$ such that the system
\[
\vert x_{0}\vert\leq X, \qquad \vert x_{0}+\zeta x_{1}+\cdots+\zeta^{k}x_{k}\vert\leq X^{-\nu}
\]
has (at least) $j$ linearly independent solutions $(x_{0},\ldots,x_{k})\in{\mathbb{Z}^{k+1}}$
for arbitrarily large $X$ resp. for all sufficiently large $X$. Then,
the identities
\[
w_{k,j}(\zeta)=\frac{1}{\widehat{\lambda}_{k,k+2-j}(\zeta)}, 
\qquad \widehat{w}_{k,j}(\zeta)=\frac{1}{\lambda_{k,k+2-j}(\zeta)}
\]
hold for $1\leq j\leq k+1$, see (1.24) in~\cite{schlei}. 
Thus, the constants $w_{k,j}(\zeta), \widehat{w}_{k,j}(\zeta)$
can be readily determined for Liouville numbers $\zeta$ involved in Section~\ref{sek4}.

\vspace{1cm}

Thanks to the referee for many structural advices!

\end{document}